\documentclass{article}

\usepackage[english]{babel}
\usepackage{amsmath}
\usepackage{amssymb}
\usepackage{amsthm}
\usepackage{color}

\newtheorem{theorem}{Theorem}[section]
\newtheorem{definition}[theorem]{Definition}
\newtheorem{lemma}[theorem]{Lemma}
\newtheorem{prop}[theorem]{Proposition}
\newtheorem{corollary}[theorem]{Corollary}
\newtheorem{remark}[theorem]{Remark}
\newtheorem{ex}[theorem]{Example}

\newcommand{\hh}{{\mathbb{H}}}
\newcommand{\cc}{{\mathbb{C}}}
\newcommand{\rr}{{\mathbb{R}}}
\newcommand{\nn}{{\mathbb{N}}}

\newcommand{\s}{{\mathbb{S}}}
\newcommand{\z}{{\mathcal{Z}}}

\title{\bf The zero sets of slice regular functions and the open mapping theorem}
\author{ Graziano Gentili\footnote{Partially supported by GNSAGA of the INdAM and by PRIN ``Propriet\`a geometriche delle variet\`a reali e complesse'' of the MIUR.} \\ 
\normalsize Dipartimento di Matematica ``U. Dini'', Universit\`a di Firenze \\ 
\normalsize Viale Morgagni 67/A, 50134 Firenze, Italy,  gentili@math.unifi.it \\
\and Caterina Stoppato$^*$\footnote{Partially supported by PRIN ``Geometria differenziale e analisi globale'' of the MIUR.} \\ 
\normalsize Dipartimento di Matematica ``U. Dini'', Universit\`a di Firenze \\ 
\normalsize Viale Morgagni 67/A, 50134 Firenze, Italy,  stoppato@math.unifi.it\\}

\date{  }

\begin{document}
\maketitle


\begin{abstract}
A new class of regular quaternionic functions, defined by power series in a natural fashion, has been introduced in \cite{advances}. Several results of the theory recall the classical complex analysis, whereas other results reflect
the peculiarity of the quaternionic structure. The recent \cite{advancesrevised} identified a larger class of domains, on which the study of regular functions is most natural and not limited to the study of quaternionic power series. 
In the present paper we extend some basic results concerning the algebraic and topological properties of the zero set to regular functions defined on these domains.
We then use these results to prove the Maximum and Minimum Modulus Principles and a 
version of the Open Mapping Theorem in this new setting.
\end{abstract}


\section{Introduction}

Let $\hh$ be the real algebra of quaternions. Its elements are of the form $q=x_0+ix_1+jx_2+kx_3$ where the $x_n$ are real, and $i$, $j$, $k$, are imaginary units (i.e. their square equals $-1$) such that $ij=-ji=k$, $jk=-kj=i$, and $ki=-ik=j.$ The richness of the theory of holomorphic functions of one complex variable inspired the study of several interesting theories of quaternionic functions during the last century. The most famous was introduced by Fueter, \cite{fueter1, fueter2}, and excellently surveyed in \cite{sudbery}. For recent work on Fueter regularity, see, e.g., \cite{librodaniele, libroshapiro} and references therein.

A different notion of regularity for quaternionic functions, inspired by Cullen \cite{cullen}, has been proposed in \cite{cras, advances}. Several classical results in complex analysis have quaternionic analogs, proven in \cite{advances} and in the subsequent papers \cite{zeros, open, survey, rigidity, poli} for regular functions on open balls $B(0,R) = \{q \in \hh : |q| <R\}$. The recent \cite{advancesrevised} identified a larger class of domains that are the quaternionic
analogs of the (complex) domains of holomorphy. The study of regular functions is most natural
on these domains, and it does not reduce to the study of quaternionic power series. In the present paper we extend to regular functions defined on these domains the results proven in \cite{zeros, open}, which concern the (algebraic and) topological structure the zero sets and the openness of regular functions. We also extend the Maximum and Minimum Modulus Principles proven in \cite{advances} and \cite{open} respectively. 

Let us present the bases of the theory, beginning with some notation. Denote by $\s$ the two-dimensional sphere of quaternion imaginary units: $\s = \{q \in \hh : q^2 =-1\}$. For all imaginary unit $I \in \s$, let $L_I = \rr + I \rr$ be the complex line through $0, 1$ and $I$. The definition of regularity given in \cite{advancesrevised} follows.

\begin{definition}\label{definition}
Let $\Omega$ be a domain in $\hh$ and let $f : \Omega \to \hh$ be a function. For all $I \in \s$, we set $\Omega_I = \Omega \cap L_I$ and denote the restriction $f_{|_{\Omega_I}}$ by $f_I$. The restriction $f_I$ is called \textnormal{holomorphic} if $f_I \in C^1(\Omega_I)$ and if
\begin{equation}
\bar \partial_I f (x+Iy) = \frac{1}{2} \left( \frac{\partial}{\partial x}+I\frac{\partial}{\partial y} \right) f_I (x+Iy)
\end{equation}
vanishes identically. The function $f$ is called \emph{slice regular} (or simply \emph{regular}) if $f_I$ is holomorphic for all $I \in \s$.
\end{definition}

For instance, a quaternionic power series $\sum_{n \in \nn} q^n a_n$ with $a_n \in \hh$ defines a regular function in its domain of convergence, which proves to be an open ball $B(0,R) = \{q \in \hh : |q| <R\}$. Conversely, the following is proven in \cite{advancesrevised}.

\begin{theorem} 
Let $f : \Omega \to \hh$ be a regular function. If $p \in \Omega \cap \rr$ and if $B = B(p,R)$ is the largest open ball centered at $p$ and included in $\Omega$, then there exist quaternions $a_n \in \hh$ such that $f(q)=\sum_{n \in \nn} (q-p)^n a_n$ for all $q \in B$. In particular, $f \in C^{\infty}(B)$.
\end{theorem}

Expanding a regular function at a non-real point $p \in \hh \setminus \rr$ is a much more delicate matter. A detailed study of quaternionic analyticity has been conducted in \cite{powerseries}. For the purpose of the present paper it suffices to know that choosing a ``bad'' domain of definition $\Omega$ may lead to regular functions that are not even continuous:

\begin{ex}
Let $\Omega$ be any neighborhood of $\s$ in $\hh$ which does not intersect the real axis (e.g. $\Omega = T(\s,r) = \{q \in \hh : d(q, \s) < r\}$ with $r<1/2$). Choose $I \in \s$ and define $f : \Omega \to \hh$ by setting $f_I = f_{-I} \equiv 1$, $f_J \equiv 0$ for all $J \in \s \setminus \{\pm I\}$. Since all constant functions are holomorphic, $f$ is regular according to Definition \ref{definition}. Clearly $f$ is not continuous at $I$ (nor at any point of $\Omega_I$).
\end{ex}

Such pathologies can be avoided by requiring the domain of definition to have certain topological and geometric properties. The first such condition is the following.

\begin{definition}
Let $\Omega$ be a domain in $\hh$, intersecting the real axis. If $\Omega_I = \Omega \cap L_I$ is a domain in $L_I \simeq \cc$ for all $I \in \s$ then we say that $\Omega$ is a \textnormal{slice domain}.
\end{definition}

The following result holds for regular functions on slice domains.

\begin{theorem}[Identity Principle] \label{identity}
Let $\Omega$ be a slice domain and let $f,g : \Omega \to \hh$ be regular. Suppose that $f$ and $g$ coincide on a subset $C$ of $\Omega_I$, for some $I \in \s$. If $C$ has an accumulation point in $\Omega_I$, then $f \equiv g$ in $\Omega$.
\end{theorem}

Another natural condition for  the domain of definition of a regular function is the following symmetry property.

\begin{definition}
A subset $C$ of $\hh$ is \textnormal{axially symmetric} if, for all $x+yI \in C$ with $x,y \in \rr, I \in \s$, the whole set $x+y\s = \{x+yJ : J \in \s\}$ is contained in $C$.
\end{definition}

For the sake of simplicity, we will call such a $C$ a \emph{symmetric} set.  It is worth to point out that symmetric slice domains play the role played by the domains of holomorphy in classical complex analysis, as proven in \cite{advancesrevised}:

\begin{theorem}[Extension Theorem]\label{circular extendability}
Let $\Omega \subseteq \mathbb{H}$ be a slice domain, and let $f: \Omega \to \mathbb{H}$ be a regular function. There exists a unique regular extension $\widetilde f: \widetilde \Omega \to \mathbb{H}$ of $f$ to the smaller symmetric slice domain  $\widetilde \Omega$ which contains $\Omega$.
\end{theorem}

As proven in \cite{advancesrevised}, generalizing \cite{open}, the distribution of the values of a regular function on each sphere $x+y\s$ contained in its domain of definition is quite special:

\begin{theorem}\label{formula} 
Let $f$ be a regular function on a symmetric slice domain $\Omega\subseteq \mathbb{H}$. For each $x+y\s \subset \Omega$ there exist constants $b(x,y), c(x,y) \in \hh$ such that for all $I \in \s$
\begin{equation}
f(x+yI) = b(x,y) + I c(x,y).
\end{equation}
\end{theorem}

In other words, the function $\s \to \hh$ mapping $I \mapsto f(x+yI)$ is affine. This immediately implies the following result, proven in \cite{advancesrevised} extending \cite{advances}.

\begin{corollary}\label{affinezeros}
Let $f$ be a regular function on a symmetric slice domain $\Omega\subseteq \mathbb{H}$ and let $x+y\s \subset \Omega$. If there exist distinct $I,J \in \s$ such that $f(x+yI)=0=f(x+yJ)$ then $f \equiv 0$ in $x+y\s$.
\end{corollary}

Furthermore, by direct computation 
$$b(x,y)=\frac{1}{2}\left[ f(x+yK)+f(x-yK)\right],\ c(x,y) = \frac{K}{2}\left[ f(x-yK)-f(x+yK)\right]$$
 for all $K\in\s$, so that $b,c$ are $C^\infty$ functions. Hence

\begin{corollary}
If $\Omega$ is a symmetric slice domain and $f$ is regular in $\Omega$, then $f \in C^{\infty}(\Omega)$.
\end{corollary}

Theorem \ref{formula} is also the basis for the following extension result (see \cite{advancesrevised}).

\begin{lemma}[Extension Lemma]\label{Extension Lemma}
Let $\Omega$ be a symmetric slice domain and choose $I\in\s$. If $f_I:\Omega_I \to L_I$ is holomorphic, then setting
\begin{equation}
f(x+yJ) =\frac{1}{2}\left[f_I(x+yI)+f_I(x-yI)\right]+J\frac{I}{2}\left[f_I(x-yI)-f_I(x+yI)\right]
\end{equation}
extends $f_I$ to a regular function $f:\Omega \to \hh$. $f$ is the unique such extension and it is denoted by ${\rm ext}(f_I)$.
\end{lemma}

The Extension Lemma \ref{Extension Lemma} is used in \cite{advancesrevised} to endow regular functions with a multiplicative operation. Its definition is recalled in Section \ref{sectionpreliminary}, together with other algebraic tools introduced in the same paper. 

The original part of this paper is structured as follows. After studying the algebraic properties of the zero set in Section \ref{sectionalgebraic}, we explore its topology and prove the following result in Section \ref{sectionstructure}:

\begin{theorem}[Structure of the Zero Set]
Let $\Omega\subseteq\mathbb{H}$ be a symmetric slice domain and let  $f:\Omega\to\mathbb{H}$ be a regular function. If $f$ does not vanish identically, then the zero set of $f$ consists of isolated points or isolated 2-spheres of the form $x+y\mathbb{S}$ (with $x,y \in \rr$ and $y \neq 0$).
\end{theorem}

\noindent Even though the statement above replicates the one established for quaternionic power series in their domain of convergence, its proof requires a different approach that relies upon extension results proven in \cite{advancesrevised}. This leads, in particular, to a stronger version of the Identity Principle: 

\begin{theorem}[Strong Identity Principle]
Let $f, g$ be regular functions on a symmetric slice domain $\Omega$.  If there exists a $2$-sphere (or a singleton) $S=x+y\mathbb{S}\subset \Omega$ such that the zeros of $f-g$ contained in $\Omega \setminus S$ accumulate to a point of $S$, then $f\equiv g$ on the whole $\Omega$.
\end{theorem}

\noindent Section \ref{sectionminimum} is devoted to the Maximum and Minimum Modulus Principles. Proving them we have to face the peculiarities of the quaternionic context. The approach is different from the one used in \cite{advances}.

\begin{theorem}[Maximum Modulus Principle]
Let $\Omega$ be a slice domain and let $f : \Omega \to \hh$ be regular. If $|f|$ has a relative maximum at $p \in \Omega$, then $f$ is constant.
\end{theorem}
\begin{theorem}[Minimum Modulus Principle]
Let $\Omega$ be a symmetric slice domain and let $f : \Omega \to \hh$ be a regular function. If $|f|$ has a local minimum point $p\in \Omega$ then either $f(p)=0$ or $f$ is constant.
\end{theorem}

\noindent As one may expect, these Principles are the main tools in the investigation of the topological properties of regular functions. In Section \ref{sectionopen} we define the \emph{degenerate set} of $f$ as the union $D_{f}$ of the 2-spheres $x+y\s$ such that $f_{|_{x+y\s}}$ is constant. It turns out that $D_{f}$ has no interior points and that $f$ is open on the rest of the domain:

\begin{theorem}[Open Mapping Theorem]
Let $f$ be a regular function on a symmetric slice domain $\Omega$ and let $D_f$ be its degenerate set. Then $f:\Omega \setminus \overline{D_f} \to \hh$ is open.
\end{theorem}

\noindent Removing the degenerate set is necessary, as shown by a counterexample. Under this point of view, the theory of quaternionic regular functions differs from that of holomorphic complex functions. This depends on the fact that the zero set of a holomorphic function is discrete, while a regular quaternionic function may vanish on a whole 2-sphere as explained above.


\section{Preliminary results}\label{sectionpreliminary}

The set of regular functions on a symmetric slice domain $\Omega$ becomes an algebra over $\rr$ when endowed with the usual addition $+$ and an appropriate multiplicative operation denoted by $*$ and called \emph{regular multiplication} (indeed, the pointwise multiplication does not preserve regularity). In the special case where the domain is a ball centered at $0$, we can make use of the power series expansion and define the $*$-multiplication by the formula
$$\left(\sum_{n \in \nn} q^n a_n\right) * \left(\sum_{n \in \nn} q^n b_n\right) = \sum_{n \in \nn} q^n \sum_{k = 0}^n a_k b_{n-k}$$
(see \cite{zeros} for details). The general definition, given in \cite{advancesrevised}, is based on the following result.

\begin{lemma}[Splitting Lemma]
Let $\Omega\subseteq\mathbb{H}$ be a symmetric slice domain and let $f:\Omega\to\mathbb{H}$ be a regular function. For any $I,J\in\mathbb{S}$, with $I\perp J$ there exist holomorphic functions $F,G: \Omega_I \to L_I$ such that for all $z \in \Omega_I$
\begin{equation}
 f_I(z)=F(z)+G(z)J.
\end{equation}
\end{lemma}

In order to define the regular product of two regular functions $f,g$ on a symmetric slice domain $\Omega$, let $I,J\in\mathbb{S}$, with $I\perp J$, and choose holomorphic functions $F,G,H,K: \Omega_I\to L_I$ such that for all $z\in \Omega_I$
\begin{equation}\label{splitted f}
 f_I(z)=F(z)+G(z)J, \qquad g_I(z)=H(z)+K(z)J.
\end{equation}
Let  $f_I*g_I:\Omega_I \to L_I$ be the holomorphic function defined by
\begin{equation}\label{f*g}
f_I*g_I(z)=[F(z)H(z)-G(z)\overline{K(\bar z)}]+[F(z)K(z)+G(z)\overline{H(\bar z)}]J.
\end{equation}
Using the Extension Lemma \ref{Extension Lemma}, the following definition is given in \cite{advancesrevised}:

\begin{definition}
Let $\Omega\subseteq\mathbb{H}$ be a symmetric slice domain and let $f,g:\Omega\to\mathbb{H}$  be regular. The function
$$f*g(q)={\rm ext}(f_I*g_I)(q)$$ 
defined as the extension of (\ref{f*g}) is called the \emph{regular product} of $f$ and $g$.
\end{definition}

\begin{remark}
The $*$-multiplication is associative, distributive but, in general, not commutative.
\end{remark}

The \emph{regular conjugate} of a power series, is defined in \cite{zeros} by the formula 
$$\left(\sum_{n \in \nn} q^n a_n\right)^c = \sum_{n \in \nn} q^n \bar a_n.$$ 
In \cite{advancesrevised} this concept is extended as follows.
\begin{definition}
Let $f$ be a regular function on a symmetric slice domain $\Omega$ and suppose $f$ splits on $\Omega_I$ as in formula (\ref{splitted f}),  $f_I(z)=F(z)+G(z)J$. We consider the holomorphic function 
\begin{equation}\label{formconj}
f_I^c(z)=\overline{F(\bar z)}-G(z)J
\end{equation} and define, according to the Extension Lemma \ref{Extension Lemma}, the \emph{regular conjugate} of $f$ by the formula
\begin{equation}
f^c(q)={\rm ext}(f_I^c)(q).
\end{equation}
\end{definition}

Furthermore, the following definition is given under the same assumptions.

\begin{definition}
The \emph{symmetrization} of $f$ is defined as
\begin{equation}
f^s = f*f^c = f^c * f.
\end{equation}
\end{definition}

In the case of power series,
$$\left(\sum_{n \in \nn} q^n a_n\right)^s = \sum_{n \in \nn} q^n \sum_{k = 0}^n a_k \bar a_{n-k},$$
where $\sum_{k = 0}^n a_k \bar a_{n-k} \in \rr$ for all $n \in \nn$. In the general case, when $f$ splits on $\Omega_I$ as in formula (\ref{splitted f}) then 
\begin{equation}\label{simmetrizzata}
f_I*f^c_I=(F(z)+G(z)J)*(\overline{F(\bar z)}-G(z)J)=F(z)\overline{F(\bar z)}+G(z)\overline{G(\bar z)}.
\end{equation}
Hence
\begin{equation}\label{f^s}
f^s(q)={\rm ext}(f_I*f^c_I)(q).
\end{equation}

Finally, in Section \ref{sectionminimum}, we will make use of the following operation to derive the Minimum Modulus Principle from the Maximum Modulus Principle.

\begin{definition}
Let $f$ be a regular function on a symmetric slice domain $\Omega$. The \emph{regular reciprocal} of $f$ is the function $f^{-*} : \Omega \setminus \z_{f^s} \to \hh$ defined by the equation
\begin{equation}
f^{-*}(q) = \frac{1}{f^{s}(q)} f^c(q)
\end{equation}
\end{definition}
 
By direct computation $f^{-*}$ is the inverse of $f$ with respect to $*$-multiplication, i.e. $f*f^{-*} = f^{-*}*f \equiv 1$ on $\Omega \setminus \z_{f^s}$.


\section{Algebraic properties of the zero set}\label{sectionalgebraic}

We now study of the correspondences among the zeros of $f$ and $g$ and those of the product $f*g$, the conjugate $f^c$ and the symmetrization $f^s$. We begin by recalling an alternative expression of the regular product $f*g$, proven in \cite{advancesrevised}.

\begin{prop}\label{formprod}
Let $\Omega\subseteq\mathbb{H}$ be a symmetric slice domain and let $f,g:\Omega\to\mathbb{H}$ be regular functions. For all $q\in\Omega$, if $f(q) = 0$ then $f*g(q) = 0$, else
\begin{equation}\label{formula prodotto}
f*g(q) = f(q)\ g(f(q)^{-1} q f(q)).
\end{equation}
\end{prop}

\begin{corollary}\label{zeriprodotto}
Let $\Omega\subseteq\mathbb{H}$ be a symmetric slice domain and let $f,g:\ \Omega\to\mathbb{H}$ be regular functions.
Then $f*g(q) = 0$ if and only if $f(q) = 0$ or $f(q) \neq 0$ and $g(f(q)^{-1} q f(q))=0$.
\end{corollary}

In particular, for each zero of $f*g$ in $S=x+y\s$ there exists a zero of $f$ or a zero of $g$ in $S$. However, \cite{zeros} presented examples of products $f*g$ whose zeros were not in one-to-one correspondence with the union of the zero sets of $f$ and $g$.
We now study the relation between the zeros of $f$ and those of $f^c$ and $f^s$. We need two preliminary steps.

\begin{lemma}\label{simmetriazerireale}
Let $\Omega\subseteq \mathbb{H}$ be a symmetric slice domain and let $f : \Omega \to \mathbb{H}$ be a regular function such that $f(\Omega_I) \subseteq L_I$ for all $I \in \s$. If $f(x_0+y_0I_0)=0$ for some $I_0\in\mathbb{S}$, then $f(x_0+y_0I)=0$ for all $I\in \mathbb{S}.$
\end{lemma}

\begin{proof}
The fact that $f(\Omega_I)\subseteq L_I$ for all $I \in \mathbb{S}$ implies that $f(x)$ is real for all $x\in \Omega\cap \mathbb{R}$. We now have a holomorphic function $f_{I_0}:\Omega_{I_0} \to L_{I_0}$ mapping $\Omega \cap \rr$ to $\rr$. By the (complex) Schwarz Reflection Principle, $f(x+yI_0)=\overline{f(x-yI_0)}$ for all $x+yI_0 \in \Omega_{I_0}$.
Since $f(x_0+y_0I_0)=0$, we conclude that $f(x_0-y_0I_0)=0$ and Corollary \ref{affinezeros} allows us to deduce the thesis.
\end{proof}

\begin{lemma}\label{simmetrizzatareale}
Let $\Omega\subseteq \mathbb{H}$ be a symmetric slice domain, let $f : \Omega \to \mathbb{H}$ be a regular function and let $f^s$ be its symmetrization. Then $f^s(\Omega_I) \subseteq L_I$ for all $I \in \s$.
\end{lemma}
\begin{proof}
It follows by direct computation from equation (\ref{simmetrizzata}).
\end{proof}

\begin{prop}\label{zericoniugata}
Let $\Omega\subseteq\mathbb{H}$ be a symmetric slice domain, let $f:\Omega\to\mathbb{H}$ be regular and choose $S=x_0+y_0\mathbb{S}\subset \Omega$. The zeros of $f$ in $S$ are in one-to-one correspondence with those of $f^c$. Furthermore, $f^s$ vanishes identically on $S$ if and only if $f^s$ has a zero in $S$, if and only if $f$ has a zero in $S$ (if and only if $f^c$ has a zero in $S$).
\end{prop}

\begin{proof}
If $q_0=x_0+y_0I_0$ is a zero of $f$ then $f^s = f*f^c$ vanishes at $q_0$ by Proposition \ref{formprod}. According to Lemmas \ref{simmetriazerireale} and \ref{simmetrizzatareale}, $f^s(x_0+y_0I) = 0$ for all $I \in \s$. 

By Corollary \ref{zeriprodotto}, the fact that $f^s(\bar q_0) = f^s(x_0-y_0I_0) = 0$ implies that either $f(\bar q_0)=0$ or $f^c(f(\bar q_0)^{-1} \bar q_0 f(\bar q_0))=0$. In the first case we conclude that $f$ vanishes identically on $S$, which implies that $f^c$ vanishes on $S$ because of formula (\ref{formconj}). In the second case, $f^c$ vanishes at the point $f(\bar q_0)^{-1} \bar q_0 f(\bar q_0) = x_0 - y_0 [f(\bar q_0)^{-1} I_0 f(\bar q_0)] \in S$. 

We have proven that if $f$ has a zero in $S$ then $f^s$ has a zero in $S$, which leads to the vanishing of $f^s$ on the whole $S$, which implies the existence of a zero of $f^c$ in $S$. Since $(f^c)^c=f$, exchanging the roles of $f$ and $f^c$ proves the thesis.
\end{proof}


\section{Topological properties of the zero set}\label{sectionstructure}

We now study the distribution of the zeros of regular functions on symmetric slice domains. In order to obtain a full characterization of the zero set of a regular function, we first deal with a special case that will be crucial in the proof of the main result.

\begin{lemma}\label{zeriseriereale}
Let $\Omega\subseteq \mathbb{H}$ be a symmetric slice domain and let $f : \Omega \to \mathbb{H}$ be a regular function such that $f(\Omega_I) \subseteq L_I$ for all $I \in \s$. If $f\not\equiv 0$, the zero set of $f$ is either empty or it is the union of isolated points (belonging to $\mathbb{R}$) and isolated 2-spheres of the type $x+y\mathbb{S}$.
\end{lemma}

\begin{proof}
We know from Lemma \ref{simmetriazerireale} that the zero set of such an $f$ consists of real points and 2-spheres of the type $x+y\mathbb{S}$. Now choose $I$ in $\mathbb{S}$ and notice that the intersection of $L_I$ with the zero set of $f$ consists of all the real zeros of $f$ and of exactly two zeros for each sphere $x+y\mathbb{S}$ on which $f$ vanishes (namely, $x+yI$ and $x-yI$). If $f\not \equiv 0$ then, by Theorem \ref{identity}, the zeros of $f$ in $L_I$ must be isolated. Hence the zero set of $f$ consists of isolated real points and isolated 2-spheres.
\end{proof}

We now state and prove the result on the topological structure of the zero set of regular functions.

\begin{theorem}[Structure of the Zero Set]\label{structurethm}
Let $\Omega\subseteq\mathbb{H}$ be a symmetric slice domain and let  $f:\Omega\to\mathbb{H}$ be a regular function. If $f$ does not vanish identically, then the zero set of $f$ consists of isolated points or isolated 2-spheres of the form $x+y\mathbb{S}$.
\end{theorem}

\begin{proof}
Consider the symmetrization $f^s$ of $f$: by Lemma \ref{simmetrizzatareale}, $f^s$ fulfills the hypotheses of Lemma \ref{zeriseriereale}. Hence the zero set of $f^s$ consists of isolated real points or isolated 2-spheres. According to Proposition \ref{zericoniugata}, the real zeros of $f$ and $f^s$ are exactly the same. Furthermore, each 2-sphere in the zero set of $f^s$ corresponds either to a 2-sphere of zeros, or to a single zero of $f$. This concludes the proof.
\end{proof}

As an immediate consequence of the previous result, we can strengthen the Identity Principle \ref{identity}.

\begin{theorem}[Strong Identity Principle]
Let $f, g$ be regular functions on a symmetric slice domain $\Omega$.  If there exists a $2$-sphere (or a singleton) $S=x+y\mathbb{S}\subset \Omega$ such that the zeros of $f-g$ contained in $\Omega \setminus S$ accumulate to a point of $S$, then $f\equiv g$ on the whole $\Omega$.
\end{theorem}


\section{The maximum and Minimum Modulus Principles}\label{sectionminimum}

The Maximum Modulus Principle is a consequence of the analogous result for holomorphic functions. Our proof uses the Identity Principle \ref{identity}, thus we must work on a slice domain.

\begin{theorem}[Maximum Modulus Principle] \label{maximum}
Let $\Omega$ be a slice domain and let $f : \Omega \to \hh$ be regular. If $|f|$ has a relative maximum at $p \in \Omega$, then $f$ is constant.
\end{theorem}

\begin{proof}
If $f(p) = 0$ then the thesis is obvious. Else we may suppose $f(p) \in \rr, f(p)>0$, possibly multiplying $f$ by $\overline{f(p)}$. Let $I,J \in \s$ be such that $p \in L_I$ and $I \perp J$; let $F,G : \Omega_I \to L_I$ be holomorphic functions such that $f_I = F+GJ$. Then $|F(p)|^2 = |f_I(p)|^2 \geq |f_I(z)|^2 = |F(z)|^2 + |G(z)|^2 \geq |F(z)|^2$ for all $z$ in a neighborhood $U_I$ of $p$ in $\Omega_I$. Hence $|F|$ has a relative maximum at $p$ and the Maximum Modulus Principle for holomorphic functions of one complex variable allows us to conclude that $F$ is constant. Namely, $F \equiv f(p)$.

Now, $|G(z)|^2 = |f_I(z)|^2-|F(z)|^2 = |f_I(z)|^2-|f_I(p)|^2 \leq  |f_I(p)|^2-|f_I(p)|^2 = 0$ for all $z \in U_I$. Hence $f_I = F \equiv f(p)$ in $U_I$. Since $\Omega$ is a slice domain, the Identity Principle \ref{identity} allows us to conclude that $f \equiv f(p)$ in $\Omega$.
\end{proof}

The Minimum Modulus Principle proven \cite{open} for power series extends to all regular functions on symmetric slice domains with a very similar proof, which  we repeat for the sake of completeness.
We first find an alternative expression of the regular reciprocal $f^{-*}$.

\begin{prop}\label{reciprocalformula}
Let $f$ be a regular function on a symmetric slice domain $\Omega$. Then, for all $q \in \Omega \setminus \z_{f^s}$,
\begin{equation}
f^{-*}(q) = \frac{1}{f(T_f(q))},
\end{equation}
where $T_f : \Omega \setminus \z_{f^s} \to \Omega \setminus \z_{f^s}$ is defined by $T_f(q) = f^c(q)^{-1} q f^c(q)$. Furthermore, $T_f$ and $T_{f^c}$ are mutual inverses so that $T_f$ is a diffeomorphism.
\end{prop}

\begin{proof}
As explained in Proposition \ref{zericoniugata}, if $f^s(q) \neq 0$ then $f^c(q)\neq 0$. Hence $T_f$ is well defined on $\Omega \setminus \z_{f^s}$. According to Proposition \ref{formprod}, $f^c(q) * g(q) = f^c(q) g(T_f(q))$. We compute:
$$f^{-*}(q) = f^s(q)^{-1} f^c(q) = \left[ f^c*f(q) \right]^{-1} f^c(q) =$$ 
$$= \left[f^c(q) f(T_f(q))\right]^{-1} f^c(q) = f(T_f(q))^{-1} f^c(q)^{-1} f^c(q) = f(T_f(q))^{-1}.$$
Moreover, $T_f : \Omega \setminus \z_{f^s} \to \hh$ maps any 2-sphere (or real singleton) $x+y\s$ to itself. In particular $T_f(\Omega \setminus \z_{f^s}) \subseteq \Omega \setminus \z_{f^s}$ (indeed, $\z_{f^s}$ is symmetric as explained in Proposition \ref{zericoniugata}).
The conjugacy operation is an involution, i.e. $(f^c)^c = f$; thus $T_{f^c}(q) = f(q)^{-1} q f(q)$. For all $q \in \Omega \setminus \z_{f^s}$, set $p = T_f(q)$ and notice that
$$T_{f^c} \circ T_f(q) = T_{f^c}(p) = f(p)^{-1} p f(p) =$$ 
$$= f(p)^{-1} \left[f^c(q)^{-1} q f^c(q)\right] f(p) = \left [f^c(q)f(p)\right]^{-1} q \left [f^c(q)f(p)\right]$$ 
where
$$f^c(q)f(p) = f^c(q) f(f^c(q)^{-1}q f^c(q)) = f^c* f (q) = f^s(q).$$ 
Hence 
$$T_{f^c} \circ T_f(q) =  f^s(q)^{-1}  q  f^s(q) = q,$$
where the last equality holds because $f^s(q)$ and $q$ commute, since they always lie in the same complex line by Lemma \ref{simmetrizzatareale}.
\end{proof}

The regular reciprocal $f^{-*}$ allows the proof of the following theorem.

\begin{theorem}[Minimum Modulus Principle]\label{minimum}
Let $\Omega$ be a symmetric slice domain and let $f : \Omega \to \hh$ be a regular function. If $|f|$ has a local minimum point $p\in \Omega$ then either $f(p)=0$ or $f$ is constant.
\end{theorem}

\begin{proof}
Consider a regular $f : \Omega \to \hh$ whose modulus has a minimum point $p \in \Omega$ with $f(p) \neq 0$. Such an $f$ does not vanish on the sphere $S = x+y\s$ through $p$. Indeed, if $f$ vanished at a point $p' \in S$ then $|f_{|_S}|$ would have a relative minimum at two distinct points: $p$ and $p'$. On the contrary, $|f_{|_S}|$ has one global minimum, one global maximum and no other extremal point: by Theorem \ref{formula}, $I \mapsto f(x+yI)$ is affine. Hence $f$ does not have zeroes in $S$, nor does $f^s$. Hence the domain $\Omega' = \Omega \setminus \z_{f^s}$ of the regular reciprocal $f^{-*}$ includes $S$. Thanks to Proposition \ref{reciprocalformula},
$$|f^{-*}(q)| = \frac{1}{|f(T_f(q))|}$$ 
for all $q \in \Omega'$. By Proposition \ref{reciprocalformula}, if $|f|$ has a minimum at $p \in x+y\s \subseteq \Omega'$ then $|f \circ T_f|$ has a minimum at $p' = T_{f^c}(p) \in \Omega'$. As a consequence, $|f^{-*}|$ has a maximum at $p'$. By the Maximum Modulus Principle \ref{maximum}, $f^{-*}$ is constant on $\Omega'$. This implies that $f$ is constant in $\Omega'$ and, thanks to the Identity Principle \ref{identity}, in the whole domain $\Omega$.
\end{proof}

As in the case of power series, the Minimum Modulus Principle is the basis for the proof of the Open Mapping Theorem.


\section{The Open Mapping Theorem}\label{sectionopen}

We are now ready to extend the Open Mapping Theorem, proven in \ref{open} for power series,
to all regular functions on symmetric slice domains. We begin with the following result.

\begin{theorem}\label{circularopen}
Let $\Omega$ be a symmetric slice domain and let $f : \Omega \to \hh$ be a non-constant regular function.  If $U$ is a symmetric open subset of $\Omega$, then $f(U)$ is open. In particular, the image $f(\Omega)$ is open.
\end{theorem}

\begin{proof}
Let $p_0 \in f(U)$. Choose $q_0 = x_0+y_0 I \in U$ such that $f(q_0)= p_0$, so that $f(q) - p_0$ has a zero on $S=x_0+y_0 \s \subseteq U$. For $r>0$, consider the symmetric neighborhood of $S$ defined by $T(S,r) = \{q \in \hh : d(q,S) < r \}$. There exists $r>0$ such that $\overline{T(S,r)} \subseteq U$ and $f(q) - p_0 \neq 0$ for all $q \in \overline{T(S,r)} \setminus S$. Let $\varepsilon >0$ be such that $|f(q) - p_0| \geq 3 \varepsilon$ for all $q$ such that $d(q, S) = r$. For all $p \in B(p_0,\varepsilon)$  and all $q$ such that $d(q, S) = r$ we get $$|f(q) -p| \geq |f(q) -p_0| - |p-p_0| \geq 3 \varepsilon -\varepsilon = 2 \varepsilon > \varepsilon \geq |p_0-p| = |f(q_0)-p|. $$ Thus $|f(q) -p|$ must have a local minimum point in $T(S,r)$. By the Minimum Modulus Principle \ref{minimum}, there exists $q \in T(S,r)$ such that $f(q) - p = 0$.
\end{proof}

\begin{definition}
Let $\Omega$ be a symmetric slice domain and let $f : \Omega \to \hh$ be a regular function. We define the \textnormal{degenerate set} of $f$ as the union $D_{f}$ of the 2-spheres $x+y\s$ (with $y \neq 0$) such that $f_{|_{x+y\s}}$ is constant.
\end{definition}

Theorem \ref{circularopen} allows the study of the topological properties of the degenerate set.

\begin{theorem}
Let $\Omega$ be a symmetric slice domain and let $f : \Omega \to \hh$ be a non-constant regular function. The degenerate set $D_f$ is closed in $\Omega\setminus \rr$ and it has empty interior.
\end{theorem}

\begin{proof}  As we saw in Theorem \ref{formula} and in the following discussion, there exist $b,c \in C^{\infty}$ such that $f(x+yI) = b(x,y)+Ic(x,y)$. Clearly, the union $\Gamma$ of the 2-spheres (or real singletons) $x+y\s$ such that $c(x,y)=0$ is closed in $\Omega$. 

If the interior of $\Gamma$ were not empty, then it would be a symmetric open set having non-open image: indeed, $f(x+yI) = b(x,y)$ for all $x+yI \in \Gamma$ and the image through $b$ of a non-empty subset of $\rr^2$ cannot be open in $\hh$. By Theorem \ref{circularopen}, $f$ would have to be constant, a contradiction with the hypothesis.

Finally, we observe that $D_{f}= \Gamma \setminus \rr$, so that $D_f$ must be closed in $\Omega\setminus \rr$ and have empty interior.
\end{proof}

We are now ready for the main result of this section.

\begin{theorem}[Open Mapping Theorem]\label{open}
Let $f$ be a regular function on a symmetric slice domain $\Omega$ and let $D_f$ be its degenerate set. Then $f:\Omega \setminus \overline{D_f} \to \hh$ is open.
\end{theorem}

\begin{proof}
Let $U$ be an open subset of $\Omega\setminus \overline{D_f}$ and let $p_0 \in f(U)$. We will show that the image $f(U)$ contains a ball $B(p_0,\varepsilon)$ with $\varepsilon > 0$. Choose $q_0 \in U$ such that $f(q_0) = p_0$. Since $U$ does not intersect any degenerate sphere, by Theorem \ref{structurethm} the point $q_0$ must be an isolated zero of the function $f(q) - p_0$. We may thus choose $r>0$ such that $\overline{B(q_0,r)} \subseteq U$ and $f(q) - p_0 \neq 0$ for all $q \in \overline{B(q_0,r)} \setminus \{ q_0 \}$. Let $\varepsilon >0$ be such that $|f(q) - p_0| \geq 3 \varepsilon$ for all $q$ such that $|q-q_0| = r$. For all $p \in B(p_0, \varepsilon)$ we get 
$$|f(q) -p| \geq |f(q) -p_0| - |p-p_0| \geq 3 \varepsilon -\varepsilon = 2 \varepsilon$$ 
for $|q-q_0| = r$, while 
$$|f(q_0)-p| = |p_0-p| \leq \varepsilon.$$ 
Thus $|f(q_0)-p| < \min_{|q-q_0| = r} |f(q) -p|$ and $|f(q) -p|$ must have a local minimum point in $B(q_0,r)$. Since $f(q) -p$ is not constant, it must vanish at the same point by Theorem \ref{minimum}. Hence there exists $q \in B(q_0,r) \subseteq U$ such that $f(q) = p$.
\end{proof}

As in \cite{open}, the non-degeneracy hypothesis cannot be removed.

\begin{ex} 
The 2-sphere $\s$ of imaginary units is degenerate for the quaternionic polynomial $f(q) = q^{-2}+1$, since $f(I) = 0$ for all $I \in \s$. We can easily prove that $f : \hh \setminus \{0\} \to \hh$ is not open by choosing an $I\in \s$ and observing that the image of the open ball $B=B(I, 1/2)$ centered at $I$ is not open. Indeed, $0\in f(B)$ and $f(B) \cap L_{J} \subseteq \rr$ when $J\in \s$ is orthogonal to $I$.
\end{ex}

This phenomenon does not arise in the complex case. Checking the proofs shows that the fact of the matter is that a regular quaternionic function may vanish on a whole 2-sphere while the zero set of a non-constant holomorphic function is discrete.


\bibliography{ZerosOpen}

\bibliographystyle{abbrv}


\end{document}